\documentclass[a4paper,11pt]{article}
\usepackage{makeidx}
\usepackage{graphicx}
\usepackage[english,activeacute]{babel}
\usepackage[latin1]{inputenc}
\usepackage{amsmath}
\usepackage{amsfonts}
\usepackage{amssymb}
\usepackage{latexsym}
\usepackage{colortbl}
\usepackage{multicol}
\usepackage{hyperref}
\usepackage{xcolor}
\usepackage{mathrsfs}

\newtheorem{theorem}{Theorem}

\newtheorem{corollary}{Corollary}

\newtheorem{lemma}{Lemma}

\newtheorem{proposition}{Proposition}

\newenvironment{proof}[1][Proof]{\noindent\textbf{#1.} }{\ \rule{0.5em}{0.5em}}
\begin{document}

\title{On second order q-difference equations for high--order Sobolev-type
q-Hermite orthogonal polynomials}
\author{{Carlos Hermoso}$^{1}$, {Edmundo
J. Huertas}$^{1}$, {Alberto Lastra %
}$^{1}$, {Anier Soria-Lorente %
}$^{2}$ \\
\\
$^{1}$Departamento de Física y Matemáticas, Universidad de Alcalá\\
Ctra. Madrid-Barcelona, Km. 33,600\\
28805 - Alcalá de Henares, Madrid, Spain\\
carlos.hermoso@uah.es, edmundo.huertas@uah.es, alberto.lastra@uah.es\\
\\
$^{2}$Departamento de Tecnología, Universidad de Granma\\
Ctra. de Bayamo-Manzanillo, Km. 17,500\\
85100 - Bayamo, Cuba\\
asorial@udg.co.cu, asorial1983@gmail.com}
\maketitle

\begin{abstract}
The q-Hermite I-Sobolev type polynomials of higher order are consider for
their study. Their hypergeometric representation is provided together with
further useful properties such as several structure relations which give
rise to a three--term recurrence relation of their elements. Two different
q-difference equations satisfied by the q-Hermite I-Sobolev type polynomials
of higher order are also established.

\vspace{0.3cm}



\end{abstract}



\section{Introduction}

\label{S1-Intro}



The \textit{$q$-Hermite I polynomials} or \textit{(continuous) $q$-Hermite
polynomials} of degree $n$, usually denoted in the literature as $H_{n}(x;q)$%
, are a family of $q$-hypergeometric polynomials introduced at the end of
the nineteenth century by L. J. Rogers, in his memoirs on expansions of
certain infinite products (see \cite{R1893-1}, \cite{R1894-2}, and \cite%
{R1895-3}). In this series of papers, and along with these polynomials,
Rogers introduced the \textit{$q$-ultraspherical polynomials} (also known as 
\textit{Rogers} or \textit{Rogers--Askey--Ismail polynomials}), and he used
both families to prove the celebrated \textit{Rogers-Ramanujan identities}.
The $q$-Hermite polynomials are also related to the \textit{Rogers-Szeg\H{o}
polynomials} (see \cite{Sz1926}), and to other important families such as
the \textit{Al-Salam--Chihara polynomials} (see \cite{AC1976}). Throughout
the twentieth century, these polynomials were studied by Szeg\H{o} (see \cite%
{Sz1926}) and Carlitz (see \cite{C1956}, \cite{C1957}, and \cite{C1972}),
and now they have received a strong impulse due to their deep involvement in
different fields, appearing in classical probability theory (\cite{BMW2005}, 
\cite{BMW2008}, \cite{BWA2010}, \cite{MS2002}), non-commutative probability
(see \cite{A2005}, \cite{B2012}, \cite{V2000}), combinatorics (see \cite%
{ISV1987}, \cite{IRS1999}, and \cite{KSZ2005}) and quantum-physics (see \cite%
{FLV1995}, \cite{FLV1997}, \cite{HW1999}, \cite{OS2008}, \cite{OR1994}). For
a recent and comprehensive survey on the subject, see \cite{S2013}.

\smallskip

The $q$-Hermite polynomials are usually defined by means of their generating
function%
\begin{equation}
\sum_{n=0}^{\infty }{H_{n}(x;q)\dfrac{t^{n}}{(q;q)_{n}}}=\prod_{n=0}^{\infty
}{\dfrac{1}{1-2xtq^{n}+t^{2}q^{2n}},}  \label{e1}
\end{equation}%
where $q$ stands for their unique parameter, for which we assume that $0<q<1$%
, which means that they belong to the class of orthogonal polynomial
solutions of certain second order $q$-difference equations, known in the
literature as the Hahn class (see \cite{H1949}, \cite{KLS2010}). A recent
study of their zeros and other interesting properties has been carried out
in \cite{RK-2021M}. This sequence of polynomials lie at the bottom of the
Askey-scheme of hypergeometric orthogonal polynomials, and they are
orthogonal with respect to the measure $d\mu =(qx,-qx;q)_{\infty }d_{q}x$.
For $q=1$ one recovers the classical Hermite polynomials and, for $q=0$ one
recovers the re-scaled Chebyshev polynomials of the second kind.

\smallskip

On the other hand, the so called \textit{Sobolev orthogonal polynomials}
refer to those families of polynomials orthogonal with respect to inner
products involving positive Borel measures supported on infinite subsets of
the real line, and also involving regular derivatives. When these
derivatives appear only on function evaluations on a finite discrete set,
the corresponding families are called \textit{Sobolev-type} or \textit{%
discrete Sobolev} orthogonal polynomial sequences. For a recent and
comprehensive survey on the subject, we refer to \cite{KLS2010} and the
references therein. At the end of the twentieth century, H. Bavinck
introduced the study of inner products involving differences (instead of
regular derivatives) in uniform lattices on the real line (see~\cite{MX2015}%
, \cite{MS2002}, \cite{OS2008}, and also \cite{OR1994} for recent results on
this topic). By analogy with the continuous case, these are also called
Sobolev-type or discrete Sobolev inner products. We also refer to~\cite%
{B1995}, \cite{B1995gen}, \cite{B1996} and \cite{HS2019}.

\smallskip

In the present work, we focus on a Sobolev type family of orthogonal
polynomials associated to the $q$-Hermite I polynomials. More precisely, we
study the sequence of $q$-Hermite I-Sobolev type polynomials of higher order 
$\{\mathbb{H}_{n}(x;q)\}_{n\geq 0}$, which are orthogonal with respect to
Sobolev-type inner product%
\begin{equation*}
\left\langle f,g\right\rangle _{\lambda} =\left\langle f,g\right\rangle
+\lambda({\mathscr D}_{q}^{j}f)(\alpha)({\mathscr D}_{q}^{j}g)(\alpha),
\label{piSobA}
\end{equation*}
where $\left\langle \cdot,\cdot\right\rangle$ stands for the inner product
associated to the orthogonality related to the monic $q-$Hermite I
polynomials, $\lambda$ stands for a fixed positive real number, $\alpha$ is
a real number outside the support of the measure $d\mu$, i.e., of absolute
value larger than one, $j\in\mathbb{N}$, and ${\mathscr D}_{q}^{j}$ stands
for the $j$-th iterate of the Euler-Jackson $q$-difference operator (see (%
\ref{eulerjackson})). To the best of our knowledge, when $\alpha $ is an arbitrary real number
outside of $[-1,1]$, i.e. the mass point is located outside the support of
the measure, the study of these higher-order $q$-discrete Sobolev
polynomials has not been carried out in the literature. Thus, the major
reason for this study is to bring together several algebraic and analytic
properties of the polynomials in $\{\mathbb{H}(x;q)\}_{n\geq 0}$ unknown so
far. Hence, the main results of the work state connection formulas relating
the families $\{H_{n}(x;q)\}_{n\geq 0}$ and $\{\mathbb{H}(x;q)\}_{n\geq 0}$,
and the hypergeometric character of the $q$-Hermite I-Sobolev type
polynomials of higher order. We also obtain a fundamental three-term
recurrence formula with rational coefficients associated to these
polynomials, and two different $q$-difference equations satisfied by them.

\smallskip

The paper is structured as follows. In Section~\ref{sec2}. we recall the
basic facts and definitions from $q$-calculus and some of the main elements
associated to the $q$-Hermite I polynomials. In Section~\ref{S3-ConnForm},
we define the $q$-Hermite I-Sobolev type polynomials of higher order and
state some properties relating them to the $q$-Hermite I polynomials via
connection formulas, and also providing their hypergeometric representation.
In Section~\ref{secpral} we state the main results of the present work,
namely the formulation of a three term recurrence relation for the $q$%
-Hermite I-Sobolev type polynomials of higher order (Theorem~\ref%
{S4-Theor3TRR-RC}), together with two different second order $q$-difference
equations satisfied by these polynomials (Theorem~\ref{sordDEqI} and Theorem~%
\ref{sordDEqII}). The paper concludes with a section of conclusions, further
remarks and examples.



\section{Preliminaries}

\label{sec2}



This preliminary section is devoted to recall the main results from $q$%
-calculus, and the definition and main properties of the $q$-Hermite I
polynomials. We have decided to include them for the sake of completeness.



\subsection{Results from $q$-calculus}



We begin by providing some background and definitions from $q$-calculus, for
the sake of completeness. We refer to~\cite{KLS2010} for further details.

Given $n\in\mathbb{N}$, the $q$-number $[n]_{q}$, is given by 
\begin{equation*}
\lbrack n]_{q}=%
\begin{cases}
\displaystyle0, & \mbox{ if }n=0, \\ 
&  \\ 
\displaystyle\frac{1-q^{n}}{1-q}=\sum_{0\leq k\leq n-1}q^{k}, & \mbox{ if }%
n\geq 1.%
\end{cases}%
\end{equation*}
It makes sense to extend the previous definition to $n\in\mathbb{Z}$ by $%
[n]_{q}=(1-q^n)/(1-q)$. From the previous definition, a $q$-analogue of the
factorial of $n$ can be stated by 
\begin{equation*}
[n]_q!=%
\begin{cases}
\displaystyle 1, & \mbox{if } n=0, \\ 
&  \\ 
\displaystyle \lbrack n] _{q}[n-1] _{q} \cdots [2]_{q}[1]_{q}, & \mbox{if }
n\geq 1.%
\end{cases}%
\end{equation*}
In addition to this, we will make use of a $q$-analogue of the Pochhammer
symbol, or shifted factorial, which is given by 
\begin{equation*}
\left( a;q\right) _{n}= 
\begin{cases}
\displaystyle 1, & \mbox{if } n=0, \\ 
&  \\ 
\displaystyle\prod_{0\leq j\leq n-1}\left( 1-aq^{j}\right) , & \mbox{if }
n\geq 1, \\ 
&  \\ 
\displaystyle (a;q)_\infty= \prod_{j\geq 0}(1-aq^{j}), & \mbox{if } n=\infty%
\mbox{ and } |a|<1.%
\end{cases}%
\end{equation*}
Let us fix the following notation 
\begin{equation*}
(a_{1},\ldots ,a_{r};q) _{k}= \prod_{1\leq j\leq r}(a_{j};q) _{k}.
\end{equation*}

Given two finite sequences of complex numbers $\left\{ a_{i}\right\}
_{i=1}^{r}$ and $\left\{ b_{j}\right\}_{i=1}^{s}$ under the assumption that $%
b_{j}\neq q^{-n}$ with $n\in \mathbb{N}$, for $j=1,2,\ldots ,s$, the basic
hypergeometric, or $q$-hypergeometric, $_{r}\phi _{s}$ series with variable $%
z$ is defined by 
\begin{equation*}
_{r}\phi _{s}\left( 
\begin{array}{c}
a_{1},a_{2},\ldots ,a_{r} \\ 
b_{1},b_{2},\ldots ,b_{s}%
\end{array}
;q,z\right) =\sum_{k\geq 0}\frac{\left( a_{1},\ldots ,a_{r};q\right) _{k}}{
\left( b_{1},\ldots ,b_{s};q\right) _{k}}\left( (-1)^{k}q^{\binom{k}{2}
}\right) ^{1+s-r}\frac{z^{k}}{\left( q;q\right) _{k}}.
\end{equation*}

Following \cite{arvesu2013first}, the $q$-falling factorial is defined by 
\begin{equation*}
\left[ s\right] _{q}^{(n)}=\frac{(q^{-s};q)_{n}}{(q-1)^{n}}q^{ns-\binom{n}{2}
},\quad n\geq 1.
\end{equation*}
We observe that $[s]_q^{(1)}$ coincides with the $q$-number $[s]_q$.
Departing from the repvious definition, one can state the $q$-analog of the
binomial coefficients (see~\cite{KLS2010}). Namely, the $q$-binomial
coefficient is defined by 
\begin{equation*}
\begin{bmatrix}
n \\ 
k%
\end{bmatrix}%
_{q}=\frac{\left( q;q\right) _{n}}{\left( q;q\right) _{k}\left( q;q\right)
_{n-k}}=\frac{\left[ n\right] _{q}!}{\left[ k\right] _{q}!\left[ n-k\right]
_{q}!},\quad n=0,1,\ldots ,n.
\end{equation*}%
where $n$ denotes a nonnegative integer.

Following~\cite{ernst2007q}, the Jackson-Hahn-Cigler $q$-subtraction is
defined by%
\begin{equation*}
\left( x\boxminus _{q}y\right) ^{n}=\sum_{0\leq k\leq n}%
\begin{bmatrix}
n \\ 
k%
\end{bmatrix}%
_{q}q^{\binom{k}{2}}(-y)^{k}x^{n-k}.
\end{equation*}
It will be useful for determining a more compact writing for the derivatives
of the reproducing kernel of the sequence of $q$-Hermite I polynomials, and
all the elements determined from it.

At this point, we define the $q$-derivative, or the Euler--Jackson $q$%
-difference operator, by 
\begin{equation}  \label{eulerjackson}
({\mathscr D}_{q}f)(z)=%
\begin{cases}
\displaystyle\frac{f(qz)-f(z)}{(q-1)z}, & \text{if}\ z\neq 0,\ q\neq 1, \\ 
&  \\ 
f^{\prime }(z), & \text{if}\ z=0,\ q=1,%
\end{cases}%
\end{equation}%
where ${\mathscr D}_{q}^{0}f=f$, ${\mathscr D}_{q}^{n}f={\mathscr D}_{q}({%
\mathscr D}_{q}^{n-1}f)$, with $n\geq 1$. We observe that 
\begin{equation*}
\lim\limits_{q\rightarrow 1}{\mathscr D}_{q}f(z)=f^{\prime }(z).
\end{equation*}

The previous $q$-analog of the derivative operator will determine two
functional equations satisfied by the Sobolev type polynomials defined in
the present work. Moreover, the $q$-derivative satisfies the following
algebraic statements: 
\begin{equation}
{\mathscr D}_{q}[f(\gamma x)]=\gamma({\mathscr D}_{q}f)(\gamma x),\quad
\forall\,\,\gamma\in\mathbb{C}.  \label{cadrule}
\end{equation}
\begin{equation}
{\mathscr D}_{q}f(z)={\mathscr D}_{q^{-1}}f(qz).  \label{DqProp}
\end{equation}
\begin{equation}
{\mathscr D}_q[f(z)g(z)]=f(qz){\mathscr D}_qg(z)+g(z){\mathscr D}_qf(z) .
\label{prodqD}
\end{equation}
\begin{equation}
{\mathscr D}_{q}({\mathscr D}_{q^{-1}}f)(z)=q^{-1}{\mathscr D}_{q^{-1}}({%
\mathscr D}_{q}f)(z).  \label{DqProOK}
\end{equation}



\subsection{$q$-Hermite I orthogonal polynomials}

\label{S11-Prelim}



After the above section recalling the main elements of $q$-calculus, we
continue by giving several aspects and properties of the $q$-Hermite I
polynomials $\{H_{n}(x;q)\}_{n\geq 0}$ .

Departing from the generating function (\ref{e1}), one has that the monic $q$%
-Hermite I polynomials can be explicitly given by%
\begin{equation}
H_{n}(x;q)=q^{\binom{n}{2}}\,_{2}\phi _{1}\left( 
\begin{array}{c}
q^{-n},x^{-1} \\ 
0%
\end{array}%
;q,-qx\right) ,  \label{ASP}
\end{equation}%
satisfying the orthogonality relation%
\begin{equation*}
\int_{-1}^{1}H_{m}(x;q)H_{n}(x;q)(qx,-qx;q)_{\infty
}d_{q}x=(1-q)(q;q)_{n}(q,-1,-q;q)_{\infty }q^{\binom{n}{2}}\delta _{m,n},
\end{equation*}
where $\delta_{m,n}$ stands for Kronecker delta. In the previous definition,
the $q$-integral is defined for every real function $f$ defined in $[-1,1]$
by 
\begin{equation*}
\int_{-1}^1f(t)d_qt=(1-q)\left(\sum_{n\ge0}f(q^{n})q^n+\sum_{n%
\ge0}f(-q^{n})q^n\right).
\end{equation*}
We observe the evaluations of $f$ make sense due to $0<q<1$, and we refer to
(1.15.7) in~\cite{KLS2010} for an expression derived from that formula.


The following statements can be derived from the previous definitions (see~%
\cite{KLS2010} for further details).

\begin{proposition}
\label{S1-Proposition11} Let $\{H_{n}(x;q)\}_{n\geq 0}$ be the sequence of $%
q $-Hermite I polynomials of degree $n$. Then following statements hold.

\begin{enumerate}
\item Recurrence relation. The recurrence relation holds for every integer $%
n\ge0$ 
\begin{equation}
xH_{n}(x;q) =H_{n+1}(x;q)+\gamma_n H_{n-1}(x;q),  \label{ReR}
\end{equation}
with initial conditions $H_{-1}(x;q)=0$ and $H_{0}(x;q)=1$. Here, $%
\gamma_n=q^{n-1}(1-q^{n})$.

\item Squared norm. For every $n\in \mathbb{N}$, 
\begin{equation*}
||H_{n}||^{2}=(1-q)(q;q)_n(q,-1,-q;q)_{\infty}q^{\binom{n}{2}}.
\end{equation*}

\item Forward shift operator. 
\begin{equation}
{\mathscr D}_{q}^{k}H_{n}(x;q) =[n]_{q}^{(k)}H_{n-k}(x;q) ,  \label{FSop}
\end{equation}
where we recall that 
\begin{equation*}
\left[ n\right] _{q}^{(k)}=\frac{(q^{-n};q)_{k}}{(q-1)^{k}}q^{kn-\binom{k}{2}
},
\end{equation*}
denotes the $q$-falling factorial.

\item Second-order $q$-difference equation. 
\begin{equation*}
\sigma(x){\mathscr D}_q{\mathscr D}_{q^{-1}}H_{n}(x;q)+\tau(x){\mathscr D}%
_{q}H_{n}(x;q)+\lambda_{n,q}H_{n}(x;q)=0,
\end{equation*}
where $\sigma(x)=x^2-1$, $\tau(x)=(1-q)^{-1}x$ and $%
\lambda_{n,q}=[n]_q([1-n]_q\sigma^{\prime \prime }/2-\tau^{\prime })$.
\end{enumerate}
\end{proposition}

It is worth mentioning that the previous $q$-difference equation appears in 
\cite{KLS2010} in an equivalent form in equation (14.28.5).



\begin{proposition}[Christoffel-Darboux formula]
\label{S1-Proposition12} Let $\{H_{n}(x;q)\}_{n\geq 0}$ be the sequence of $%
q $-Hermite I polynomials. If we denote the $n$-th reproducing kernel by 
\begin{equation*}
K_{n,q}(x,y)=\sum_{k=0}^{n}\frac{H_{k}(x;q)H_{k}(y;q)}{||H_{k}||^{2}}.
\end{equation*}
Then, for all $n\in \mathbb{N}$, it holds that 
\begin{equation}
K_{n,q}(x,y)=\frac{H_{n+1}(x;q) H_{n}(y;q) -H_{n+1}(y;q)H_{n}(x;q) }{\left(
x-y\right) ||H_{n}||^{2}}.  \label{CDarb}
\end{equation}
\end{proposition}



Let us fix the following notation for the partial $q$-derivatives of $%
K_{n,q}(x,y)$: 
\begin{eqnarray*}
K_{n,q}^{(i,j)}(x,y) &=&{\mathscr D}_{q,y}^{j}({\mathscr D}%
_{q,x}^{i}K_{n,q}(x,y)) \\
&=&\sum_{k=0}^{n}\frac{{\mathscr D}_{q}^{i}H_{k}(x;q){\mathscr D}%
_{q}^{j}H_{k}(y;q)}{||H_{k}||^{2}}.
\end{eqnarray*}%
Christoffel-Darboux formula allow us to give a more precise writing for the
derivatives of the $n$-th reproducing kernel associated to the sequence of $%
q $-Hermite I polynomials in the next result. Its proof follows an analogous
technique as that of Lemma 1~\cite{hermoso2020second}. Therefore, we omit
the proof.


\begin{proposition}
\label{S1-LemmaKernel0j} Let $\{H_{n}(x;q)\}_{n\geq 0}$ be the sequence of $%
q $-Hermite I polynomials of degree $n$. Then following statements hold, for
all $n,j\in \mathbb{N}$ one has 
\begin{equation}
K_{n-1,q}^{(0,j)}(x,y)={\mathcal{A}}_{n}^{(j)}(x,y)H_{n}(x;q)+{\mathcal{B}}%
_{n}^{(j)}(x,y)H_{n-1}(x;q),  \label{Kernel0j}
\end{equation}
where 
\begin{equation*}
{\mathcal{A}}_{n}^{(j)}(x,y)=\frac{\left[ j\right] _{q}!}{||H_{n-1}||^{2}
\left( x\boxminus _{q}y\right) ^{j+1}}\sum_{k=0}^{j}\frac{{\mathscr D}%
_{q}^{k}H_{n-1}(y;q)}{\left[ k\right] _{q}!}(x\boxminus _{q}y)^{k},
\end{equation*}
and 
\begin{equation*}
{\mathcal{B}}_{n}^{(j)}(x,y)=-\frac{\left[ j\right] _{q}!}{%
||H_{n-1}||^{2}\left( x\boxminus _{q}y\right) ^{j+1}}\sum_{k=0}^{j}\frac{{%
\mathscr D}_{q}^{k}H_{n}(y;q)}{\left[ k\right] _{q}!}(x\boxminus _{q}y)^{k}.
\end{equation*}
\end{proposition}

\begin{proposition}
\label{S1-LemmaKerneli2} Let $\{H_{n}(x;q)\}_{n\geq 0}$ be the sequence of $%
q $-Hermite I polynomials of degree $n$. Then following statements hold, for
all $n,j\in \mathbb{N}$, 
\begin{eqnarray}
K_{n-1,q}^{(1,j)}(x,y) &=&{\mathcal{C}}_{1,n}(x,y)H_{n}(x;q)+{\mathcal{D}}%
_{1,n}(x,y)H_{n-1}(x;q),  \label{kernel1j} \\
K_{n-1,q}^{(2,j)}(x,y) &=&{\mathcal{C}}_{2,n}(x,y)H_{n}(x;q)+{\mathcal{D}}%
_{2,n}(x,y)H_{n-1}(x;q),  \label{kernel2j}
\end{eqnarray}%
where%
\begin{equation*}
{\mathcal{C}}_{1,n}(x,y)={\mathscr D}_{q}{\mathcal{A}}%
_{n}^{(j)}(x,y)-[n-1]_{q}\gamma _{n-1}^{-1}{\mathcal{B}}_{n}^{(j)}(qx,y),
\end{equation*}%
\begin{equation*}
{\mathcal{D}}_{1,n}(x,y)=[n]_{q}{\mathcal{A}}_{n}^{(j)}(qx,y)+[n-1]_{q}%
\gamma _{n-1}^{-1}x{\mathcal{B}}_{n}^{(j)}(qx,y)+{\mathscr D}_{q}{\mathcal{B}%
}_{n}^{(j)}(x,y),
\end{equation*}%
\begin{equation*}
{\mathcal{C}}_{2,n}(x,y)={\mathscr D}_{q}{\mathcal{C}}_{1,n}(x,y)-[n-1]_{q}%
\gamma _{n-1}^{-1}{\mathcal{D}}_{1,n}(qx,y),
\end{equation*}%
and 
\begin{equation*}
{\mathcal{D}}_{2,n}(x,y)=[n]_{q}{\mathcal{C}}_{1,n}(qx,y)+[n-1]_{q}\gamma
_{n-1}^{-1}x{\mathcal{D}}_{1,n}(qx,y)+{\mathscr D}_{q}{\mathcal{D}}%
_{1,n}(x,y).
\end{equation*}
\end{proposition}

\begin{proof}
Applying the $q$-derivative operator ${\mathscr D}_{q}$ with respect to $x$
variable to (\ref{Kernel0j}), together with the property \eqref{prodqD}
yields%
\begin{equation*}
K_{n-1,q}^{(1,j)}(x,y)={\mathcal{A}}_{n}^{(j)}(qx){\mathscr D}%
_{q}H_{n}(x;q)+H_{n}(x;q){\mathscr D}_{q}{\mathcal{A}}_{n}^{(j)}(x)
\end{equation*}%
\begin{equation*}
+{\mathcal{B}}_{n}^{(j)}(qx){\mathscr D}_{q}H_{n-1}(x;q)+H_{n-1}(x;q){%
\mathscr D}_{q}{\mathcal{B}}_{n}^{(j)}(x).
\end{equation*}
Using \eqref{prodqD}, \eqref{ReR} and \eqref{FSop} we deduce \eqref{kernel1j}%
. We obtain \eqref{kernel2j} in an analogous way departing from (\ref%
{kernel1j}) instead of (\ref{Kernel0j}).
\end{proof}



\section{Connection formulas and hypergeometric representation}

\label{S3-ConnForm}



In this section we introduce the $q$-Hermite I-Sobolev type polynomials of
higher order $\{\mathbb{H}_{n}(x;q)\}_{n\geq 0}$, which are orthogonal with
respect to the Sobolev-type inner product%
\begin{equation}
\left\langle f,g\right\rangle _{\lambda}
=\int_{-1}^{1}f(x;q)g(x;q)(qx,-qx;q)_{\infty }d_{q}x +\lambda({\mathscr D}%
_{q}^{j}f)(\alpha;q)({\mathscr D}_{q}^{j}g)(\alpha;q),  \label{piSob}
\end{equation}
where $\alpha\in\mathbb{R}\setminus[-1,1]$, $\lambda\in\mathbb{R}^{+}$ and $%
j\in\mathbb{N}$. In addition, we express the $q$-Hermite I-Sobolev type
polynomials of higher order $\{\mathbb{H}_{n}(x;q)\}_{n\geq 0}$ in terms of
the $q$-Hermite polynomials $\{H_{n}(x;q)\}_{n\geq 0}$, the kernel
polynomials and their corresponding derivatives. Moreover, we obtain a
representation of the proposed polynomials as hypergeometric functions.

\begin{proposition}
Let $\{\mathbb{H}_{n}(x;q)\}_{n\geq 0}$ be the sequence of $q$-Hermite
I-Sobolev type polynomials of degree $n$. Then, the following statements
hold for $n\ge1$: 
\begin{equation}
\mathbb{H}_{n}(x;q)=H_{n}(x;q)-\lambda \frac{[n]_{q}^{(j)}H_{n-j}(\alpha;q)}{%
1+\lambda K_{n-1,q}^{(j,j)}(\alpha,\alpha)}K_{n-1,q}^{(0,j)}(x,\alpha).
\label{ConxF1}
\end{equation}
\end{proposition}

\begin{proof}
Taking into account the Fourier expansion 
\begin{equation*}
\mathbb{H}_{n}(x;q)=H_{n}(x;q)+\sum_{0\leq k\leq n-1}a_{n,k}H_{k}(x;q),
\end{equation*}
one can apply the property \eqref{piSob} and consider the orthogonality
properties for $H_{n}(x;q)$. Therefore, the coefficients in the previous
expansion are given by 
\begin{equation*}
a_{n,k}=-\frac{\lambda {\mathscr D}_{q}^{j}\mathbb{H}_{n}(\alpha;q){\mathscr %
D}_{q}^{j}H_{k}(\alpha;q)}{||H_{k}||^{2}},\quad 0\leq k\leq n-1.
\end{equation*}
Thus 
\begin{equation*}
\mathbb{H}_{n}(x;q)=H_{n}(x;q)-\lambda {\mathscr D}_{q}^{j}\mathbb{H}%
_{n}(\alpha;q)K_{n-1,q}^{(0,j)}(x,\alpha).
\end{equation*}
Applying the operator ${\mathscr D}_{q}^j$ to the previous equation we
obtain from (\ref{FSop}) that 
\begin{equation*}
{\mathscr D}_{q}^j\mathbb{H}_{n}(x;q)=[n]_{q}^{(j)}H_{n-j}(x;q)-\lambda {%
\mathscr D}_{q}^{j}\mathbb{H}_{n}(\alpha;q)K_{n-1,q}^{(j,j)}(x,\alpha).
\end{equation*}
which entails after evaluation at $x=\alpha$ that 
\begin{equation*}
{\mathscr D}_{q}^{j}\mathbb{H}_{n}(\alpha;q)=\frac{[n]_{q}^{(j)}H_{n-j}(%
\alpha;q)}{1+\lambda K_{n-1,q}^{(j,j)}(\alpha,\alpha)}.
\end{equation*}
Therefore, we obtain \eqref{ConxF1}.
\end{proof}

As a consequence, we have the following results.

\begin{corollary}
\label{DqDq2HS} Let $\{\mathbb{H}_{n}(x;q)\}_{n\geq 0}$ be the sequence of $%
q $-Hermite I-Sobolev type polynomials of degree $n$. Then, the following
statements hold, 
\begin{equation*}
{\mathscr D}_{q}\mathbb{H}_{n}(x;q)=[n]_{q}H_{n-1}(x;q)-\lambda \frac{%
[n]_{q}^{(j)}H_{n-j}(\alpha;q)}{1+\lambda K_{n-1,q}^{(j,j)}(\alpha,\alpha)}%
K_{n-1,q}^{(1,j)}(x,\alpha).  \label{ConxF3}
\end{equation*}
and 
\begin{equation*}
{\mathscr D}_{q}^2\mathbb{H}_{n}(x;q)=[n]_{q}^{(2)}H_{n-2}(x;q)-\lambda 
\frac{[n]_{q}^{(j)}H_{n-j}(\alpha;q)}{1+\lambda
K_{n-1,q}^{(j,j)}(\alpha,\alpha)}K_{n-1,q}^{(2,j)}(x,\alpha).  \label{ConxF2}
\end{equation*}
\end{corollary}

\begin{lemma}
Let $\{\mathbb{H}_{n}(x;q)\}_{n\geq 0}$ be the sequence of $q$-Hermite
I-Sobolev type polynomials of degree $n$. Then, the following statements
hold, 
\begin{equation}
\mathbb{H}_{n}(x;q)={\mathcal{E}}_{1,n}(x)H_{n}(x;q)+{\mathcal{F}}%
_{1,n}(x)H_{n-1}(x;q),  \label{ConexF_I}
\end{equation}
where 
\begin{equation*}
{\mathcal{E}}_{1,n}(x)=1-\lambda \frac{[n]_{q}^{(j)}H_{n-j}(\alpha;q)}{%
1+\lambda K_{n-1,q}^{(j,j)}(\alpha,\alpha)}{\mathcal{A}}_n^{(j)}(x,\alpha),
\end{equation*}
and 
\begin{equation*}
{\mathcal{F}}_{1,n}(x)=-\lambda \frac{[n]_{q}^{(j)}H_{n-j}(\alpha;q)}{%
1+\lambda K_{n-1,q}^{(j,j)}(\alpha,\alpha)}{\mathcal{B}}_n^{(j)}(x,\alpha).
\end{equation*}
\end{lemma}

\begin{proof}
From \eqref{ConxF1} and Proposition \ref{S1-LemmaKernel0j} we conclude the
result.
\end{proof}

On the other hand, from the previous Lemma and recurrence relation (\ref{ReR}%
) we have the following result 
\begin{equation}
\mathbb{H}_{n-1}(x;q)={\mathcal{E}}_{2,n}(x)H_{n}(x;q)+{\mathcal{F}}%
_{2,n}(x)H_{n-1}(x;q),  \label{ConexF_II}
\end{equation}%
where 
\begin{equation*}
{\mathcal{E}}_{2,n}(x)=-\frac{{\mathcal{F}}_{1,n-1}(x)}{\gamma _{n-1}},
\end{equation*}%
and 
\begin{equation*}
{\mathcal{F}}_{2,n}(x)={\mathcal{E}}_{1,n-1}(x)-x{\mathcal{E}}_{2,n}(x).
\end{equation*}

\begin{lemma}
\label{detxi} Let $\{\mathbb{H}_{n}(x;q)\}_{n\geq 0}$ be the sequence of $q $%
-Hermite I-Sobolev type polynomials of degree $n$. Then, the following
statements hold, 
\begin{equation}
\Xi_{1,n}(x)H_{n}(x;q)=%
\begin{vmatrix}
\mathbb{H}_{n}(x;q) & \mathbb{H}_{n-1}(x;q) \\ 
{\mathcal{F}}_{1,n}(x) & {\mathcal{F}}_{2,n}(x)%
\end{vmatrix}%
,  \label{ConexF_III}
\end{equation}
and 
\begin{equation}
\Xi_{1,n}(x)H_{n-1}(x;q)=-%
\begin{vmatrix}
\mathbb{H}_{n}(x;q) & \mathbb{H}_{n-1}(x;q) \\ 
{\mathcal{E}}_{1,n}(x) & {\mathcal{E}}_{2,n}(x)%
\end{vmatrix}%
,  \label{ConexF_IV}
\end{equation}
where 
\begin{equation*}
\Xi_{1,n}(x)=%
\begin{vmatrix}
{\mathcal{E}}_{1,n}(x) & {\mathcal{E}}_{2,n}(x) \\ 
{\mathcal{F}}_{1,n}(x) & {\mathcal{F}}_{2,n}(x)%
\end{vmatrix}%
.
\end{equation*}
\end{lemma}

\begin{proof}
Let us multiply \eqref{ConexF_I} by ${\mathcal{F}}_{2,n}(x)$ and %
\eqref{ConexF_II} by $-{\mathcal{F}}_{1,n}(x)$. Adding and simplifying the
resulting equations, we deduce \eqref{ConexF_III}. In addition, we can
proceed analogously to get \eqref{ConexF_IV}.
\end{proof}

Finally, we will focus our attention in the representation of $\mathbb{H}%
_{n}(x;q)$ as hypergeometric functions. We omit the details of the proof
that follows a similar analysis to that carried out in \cite%
{costas2018analytic,hermoso2020second}.



\begin{proposition}[Hypergeometric character]
\label{S3-Theorem31} The q-Hermite I-Sobolev type polynomials of higher
order $\mathbb{H}_{n}(x;q)\}_{n\geq 0}$ , have the following hypergeometric
representation:%
\begin{equation}
\mathbb{H}_{n}(x;q)=-\frac{{\mathcal{F}}_{1,n}(x)(1-\psi _{n}(x)q^{-1})q^{%
\binom{n}{2}-n+2}}{[n]_{q}\psi _{n}(x)(1-q)}\times \,_{3}\phi _{2}\left( 
\begin{array}{c}
q^{-n},x^{-1},\psi _{n}(x) \\ 
0,\psi _{n}(x)q^{-1}%
\end{array}%
;q,-qx\right)  \label{ASPTSHR}
\end{equation}
where $\psi _{n}(x)=((1-q)\vartheta _{n}(x)+1)^{-1}$ and 
\begin{equation*}
\vartheta _{n}(x)=-\frac{q^{n-2}[n]_{q}{\mathcal{E}}_{1,n}(x)}{{\mathcal{F}}%
_{1,n}(x)}-[n-1]_{q}.
\end{equation*}
\end{proposition}

\begin{lemma}
\label{DqHS} Let $\{\mathbb{H}_{n}(x;q)\}_{n\geq 0}$ be the sequence of $q$%
-Hermite I-Sobolev type polynomials of degree $n$. Then, the following
statements hold, 
\begin{equation}
{\mathscr D}_{q}\mathbb{H}_{n}(x;q)={\mathcal{E}}_{3,n}(x)H_{n}(x;q)+{%
\mathcal{F}}_{3,n}(x)H_{n-1}(x;q),  \label{DqHsE3F3}
\end{equation}
where 
\begin{equation*}
{\mathcal{E}}_{3,n}(x)=-\lambda\frac{[n]_{q}^{(j)}H_{n-j}(\alpha;q)}{%
1+\lambda K_{n-1,q}^{(j,j)}(\alpha,\alpha)}{\mathcal{C}}_{1,n}(x,\alpha),
\end{equation*}
and 
\begin{equation*}
{\mathcal{F}}_{3,n}(x)=[n]_{q}-\lambda\frac{[n]_{q}^{(j)}H_{n-j}(\alpha;q)}{%
1+\lambda K_{n-1,q}^{(j,j)}(\alpha,\alpha)}{\mathcal{D}}_{1,n}(x,\alpha).
\end{equation*}
\end{lemma}



\begin{proof}
From Corollary \ref{DqDq2HS} and Proposition \ref{S1-LemmaKerneli2}, we
deduce the desired result.
\end{proof}



\begin{proposition}
\label{STHST} The $q$-Hermite I-Sobolev type polynomials $\mathbb{H}%
_{n}(x;q) $ of degree $n$ satisfy the following structure relation, 
\begin{equation*}
\Xi_{1,n}(x){\mathscr D}_{q}\mathbb{H}_{n}(x;q)={\mathcal{E}}_{4,n}(x)%
\mathbb{H}_{n}(x;q)+{\mathcal{F}}_{4,n}(x)\mathbb{H}_{n-1}(x;q),
\end{equation*}
where 
\begin{equation*}
{\mathcal{E}}_{4,n}(x)=-%
\begin{vmatrix}
{\mathcal{E}}_{2,n}(x) & {\mathcal{E}}_{3,n}(x) \\ 
{\mathcal{F}}_{2,n}(x) & {\mathcal{F}}_{3,n}(x)%
\end{vmatrix}%
,
\end{equation*}
and 
\begin{equation*}
{\mathcal{F}}_{4,n}(x)=%
\begin{vmatrix}
{\mathcal{E}}_{1,n}(x) & {\mathcal{E}}_{3,n}(x) \\ 
{\mathcal{F}}_{1,n}(x) & {\mathcal{F}}_{3,n}(x)%
\end{vmatrix}%
.
\end{equation*}
\end{proposition}



\begin{proof}
Using Lemma \ref{detxi} and Lemma \ref{DqHS}, respectively, we deduce that%
\begin{equation*}
\begin{vmatrix}
\mathbb{H}_{n}(x;q) & \mathbb{H}_{n-1}(x;q) \\ 
{\mathcal{F}}_{1,n}(x) & {\mathcal{F}}_{2,n}(x)%
\end{vmatrix}%
{\mathcal{E}}_{3,n}(x)-%
\begin{vmatrix}
\mathbb{H}_{n}(x;q) & \mathbb{H}_{n-1}(x;q) \\ 
{\mathcal{E}}_{1,n}(x) & {\mathcal{E}}_{2,n}(x)%
\end{vmatrix}%
{\mathcal{F}}_{3,n}(x)
\end{equation*}%
\begin{eqnarray*}
&&={\mathcal{E}}_{3,n}(x){\mathcal{F}}_{2,n}(x)\mathbb{H}_{n}(x;q)-{\mathcal{%
E}}_{3,n}(x){\mathcal{F}}_{1,n}(x)\mathbb{H}_{n-1}(x;q) \\
&&\qquad \qquad \qquad \qquad \qquad \qquad -{\mathcal{E}}_{2,n}(x){\mathcal{%
F}}_{3,n}(x)\mathbb{H}_{n}(x;q)+{\mathcal{E}}_{1,n}(x){\mathcal{F}}_{3,n}(x)%
\mathbb{H}_{n-1}(x;q)
\end{eqnarray*}%
\begin{equation*}
=-%
\begin{vmatrix}
{\mathcal{E}}_{2,n}(x) & {\mathcal{E}}_{3,n}(x) \\ 
{\mathcal{F}}_{2,n}(x) & {\mathcal{F}}_{3,n}(x)%
\end{vmatrix}%
\mathbb{H}_{n}(x;q)+%
\begin{vmatrix}
{\mathcal{E}}_{1,n}(x) & {\mathcal{E}}_{3,n}(x) \\ 
{\mathcal{F}}_{1,n}(x) & {\mathcal{F}}_{3,n}(x)%
\end{vmatrix}%
\mathbb{H}_{n-1}(x;q).
\end{equation*}
This concludes the result.
\end{proof}



\begin{lemma}
\label{Dq2HS} Let $\{\mathbb{H}_{n}(x;q)\}_{n\geq 0}$ be the sequence of $q$%
-Hermite I-Sobolev type polynomials of degree $n$. Then, the following
statements hold 
\begin{equation}
{\mathscr D}_{q}^{2}\mathbb{H}_{n}(x;q)={\mathcal{E}}_{5,n}(x)H_{n}(x;q)+{%
\mathcal{F}}_{5,n}(x)H_{n-1}(x;q),  \label{Dq2Hs}
\end{equation}%
where 
\begin{equation*}
{\mathcal{E}}_{5,n}(x)=-[n]_{q}^{(2)}\gamma _{n-1}^{-1}-\lambda \frac{%
\lbrack n]_{q}^{(j)}H_{n-j}(\alpha ;q)}{1+\lambda K_{n-1,q}^{(j,j)}(\alpha
,\alpha )}{\mathcal{C}}_{2,n}(x,\alpha ),
\end{equation*}%
and%
\begin{equation*}
{\mathcal{F}}_{5,n}(x)=[n]_{q}^{(2)}\gamma _{n-1}^{-1}x-\lambda \frac{%
\lbrack n]_{q}^{(j)}H_{n-j}(\alpha ;q)}{1+\lambda K_{n-1,q}^{(j,j)}(\alpha
,\alpha )}{\mathcal{D}}_{2,n}(x,\alpha ).
\end{equation*}
\end{lemma}



\begin{proof}
It is straightforward to check the result departing from Corollary \ref%
{DqDq2HS}, the recurrence relation (\ref{ReR}) and Proposition \ref%
{S1-LemmaKerneli2}, deducing \eqref{Dq2Hs}.
\end{proof}



\begin{proposition}
\label{2SR} The $q$-Hermite I-Sobolev type polynomials $\mathbb{H}_{n}(x;q)$
of degree $n$ satisfy the following relation, 
\begin{equation*}
\Xi_{1,n}(x){\mathscr D}_{q}^2\mathbb{H}_{n}(x;q)={\mathcal{E}}_{6,n}(x)%
\mathbb{H}_{n}(x;q)+{\mathcal{F}}_{6,n}(x)\mathbb{H}_{n-1}(x;q),
\end{equation*}
where 
\begin{equation*}
{\mathcal{E}}_{6,n}(x)=-%
\begin{vmatrix}
{\mathcal{E}}_{2,n}(x) & {\mathcal{E}}_{5,n}(x) \\ 
{\mathcal{F}}_{2,n}(x) & {\mathcal{F}}_{5,n}(x)%
\end{vmatrix}%
,
\end{equation*}
and 
\begin{equation*}
{\mathcal{F}}_{6,n}(x)=%
\begin{vmatrix}
{\mathcal{E}}_{1,n}(x) & {\mathcal{E}}_{5,n}(x) \\ 
{\mathcal{F}}_{1,n}(x) & {\mathcal{F}}_{5,n}(x)%
\end{vmatrix}%
.
\end{equation*}
\end{proposition}



\begin{proof}
From the application of Lemma \ref{detxi} and Lemma \ref{Dq2HS} we reach the
conclusion.
\end{proof}



\section{Recurrence relation and $q$-difference equations}

\label{secpral}



This main section of the present work states a three-term recurrence
relation for the sequence of polynomials $\{\mathbb{H}_{n}(x;q)\}_{n\geq 0}$%
, together with the establishment of two $q$-difference equations satisfied
by the polynomials in $\{\mathbb{H}_{n}(x;q)\}_{n\geq 0}$.



\begin{theorem}
\label{S4-Theor3TRR-RC}Let $\{\mathbb{H}_{n}(x;q)\}_{n\geq 0}$ be the
sequence of $q$-Hermite I-Sobolev type polynomials of degree $n$. Then, $%
\mathbb{H}_{n}(x;q)$ satisfies the following three-term recurrence
relations, 
\begin{equation*}
\Xi_{2,n}(x)\mathbb{H}_{n+1}(x;q)=\alpha_{n}(x)\mathbb{H}_{n}(x;q)+%
\beta_{n}(x)\mathbb{H}_{n-1}(x;q),
\end{equation*}
where 
\begin{equation*}
\Xi_{2,n}(x)_{n}(x)=\Xi_{1,n}(x){\mathcal{E}}_{4,n+1}(x),
\end{equation*}
\begin{equation*}
\alpha_{n}(x)=\Xi_{1,n+1}(x){\mathcal{E}}_{8,n}(x)-\Xi_{1,n}(x){\mathcal{F}}%
_{4,n+1}(x),
\end{equation*}
and 
\begin{equation*}
\beta_{n}(x)=\Xi_{1,n+1}(x){\mathcal{F}}_{8,n}(x).
\end{equation*}
\end{theorem}



\begin{proof}
Shifting the index in \eqref{DqHsE3F3} from $n$ to $n+1$ and using the
recurrence relation \eqref{ReR}, yields 
\begin{equation*}
{\mathscr D}_{q}\mathbb{H}_{n+1}(x;q)={\mathcal{E}}_{7,n}(x)H_{n}(x;q)+{%
\mathcal{F}}_{7,n}(x)H_{n-1}(x;q),
\end{equation*}%
where 
\begin{equation*}
{\mathcal{E}}_{7,n}(x)=x{\mathcal{E}}_{3,n+1}(x)+{\mathcal{F}}%
_{3,n+1}(x),\quad \mbox{and}\quad {\mathcal{F}}_{7,n}(x)=-\gamma _{n}{%
\mathcal{E}}_{3,n+1}(x).
\end{equation*}%
Then, using Lemma \ref{detxi}, we deduce 
\begin{equation}
\Xi _{1,n}(x){\mathscr D}_{q}\mathbb{H}_{n+1}(x;q)={\mathcal{E}}_{8,n}(x)%
\mathbb{H}_{n}(x;q)+{\mathcal{F}}_{8,n}(x)\mathbb{H}_{n-1}(x;q),
\label{DqHsnm1}
\end{equation}%
where 
\begin{equation*}
{\mathcal{E}}_{8,n}(x)=-%
\begin{vmatrix}
{\mathcal{E}}_{2,n}(x) & {\mathcal{E}}_{7,n}(x) \\ 
{\mathcal{F}}_{2,n}(x) & {\mathcal{F}}_{7,n}(x)%
\end{vmatrix}%
,
\end{equation*}%
and 
\begin{equation*}
{\mathcal{F}}_{8,n}(x)=%
\begin{vmatrix}
{\mathcal{E}}_{1,n}(x) & {\mathcal{E}}_{7,n}(x) \\ 
{\mathcal{F}}_{1,n}(x) & {\mathcal{F}}_{7,n}(x)%
\end{vmatrix}%
.
\end{equation*}%
And the other hand, from Proposition \ref{STHST} we have%
\begin{equation*}
\Xi _{1,n+1}(x)\Xi _{1,n}(x){\mathscr D}_{q}\mathbb{H}_{n+1}(x;q)=
\end{equation*}%
\begin{equation*}
\Xi _{1,n}(x){\mathcal{E}}_{4,n+1}(x)\mathbb{H}_{n+1}(x;q)+\Xi _{1,n}(x){%
\mathcal{F}}_{4,n+1}(x)\mathbb{H}_{n}(x;q).
\end{equation*}
Finally, substituting \eqref{DqHsnm1} in the previous expression, we get the
desired result.
\end{proof}



\begin{theorem}[Second order difference equation, I]
\label{sordDEqI} Let $\{\mathbb{H}_{n}(x;q)\}_{n\geq 0}$ be the sequence of $%
q$-Hermite I-Sobolev type polynomials defined by \eqref{ASPTSHR}. Then, the
following statement holds, 
\begin{equation}
{\mathcal{R}}_{n}(x){\mathscr D}_{q}^{2}\mathbb{H}_{n}(x;q)+{\mathcal{S}}%
_{n}(x){\mathscr D}_{q}\mathbb{H}_{n}(x;q)+{\mathcal{T}}_{n}(x)\mathbb{H}%
_{n}(x;q)=0,\quad n\geq 0,  \label{SDEq1}
\end{equation}
where 
\begin{equation*}
{\mathcal{R}}_{n}(x)={\mathcal{F}}_{4,n}(x)\Xi_{1,n}(x),
\end{equation*}
\begin{equation*}
{\mathcal{S}}_{n}(x) =-{\mathcal{F}}_{6,n}(x)\Xi_{1,n}(x),
\end{equation*}
and 
\begin{equation*}
{\mathcal{T}}_{n}(x) ={\mathcal{E}}_{4,n}(x){\mathcal{F}}_{6,n}(x)-{\mathcal{%
E}}_{6,n}(x){\mathcal{F}}_{4,n}(x).
\end{equation*}
\end{theorem}



\begin{proof}
We have from Proposition \ref{STHST} that 
\begin{equation}
{\mathcal{F}}_{4,n}(x)\mathbb{H}_{n-1}(x;q)=\Xi _{1,n}(x){\mathscr D}_{q}%
\mathbb{H}_{n}(x;q)-{\mathcal{E}}_{4,n}(x)\mathbb{H}_{n}(x;q).  \label{f4n}
\end{equation}%
The application of Proposition \ref{2SR} yields%
\begin{equation*}
{\mathcal{F}}_{4,n}(x)\Xi _{1,n}(x){\mathscr D}_{q}^{2}\mathbb{H}_{n}(x;q)=
\end{equation*}
\begin{equation*}
{\mathcal{E}}_{6,n}(x){\mathcal{F}}_{4,n}(x)\mathbb{H}_{n}(x;q)+{\mathcal{F}}%
_{6,n}(x){\mathcal{F}}_{4,n}(x)\mathbb{H}_{n-1}(x;q).
\end{equation*}%
Then, from \eqref{f4n} we get%
\begin{equation*}
{\mathcal{F}}_{4,n}(x)\Xi _{1,n}(x){\mathscr D}_{q}^{2}\mathbb{H}_{n}(x;q)
\end{equation*}
\begin{equation*}
={\mathcal{E}}_{6,n}(x){\mathcal{F}}_{4,n}(x)\mathbb{H}_{n}(x;q)+{\mathcal{F}%
}_{6,n}(x)[\Xi _{1,n}(x){\mathscr D}_{q}\mathbb{H}_{n}(x;q)-{\mathcal{E}}%
_{4,n}(x)\mathbb{H}_{n}(x;q)].
\end{equation*}%
Thus%
\begin{equation*}
{\mathcal{F}}_{4,n}(x)\Xi _{1,n}(x){\mathscr D}_{q}^{2}\mathbb{H}_{n}(x;q)-{%
\mathcal{E}}_{6,n}(x){\mathcal{F}}_{4,n}(x)\mathbb{H}_{n}(x;q)
\end{equation*}
\begin{equation*}
-{\mathcal{F}}_{6,n}(x)\Xi _{1,n}(x){\mathscr D}_{q}\mathbb{H}_{n}(x;q)+{%
\mathcal{E}}_{4,n}(x){\mathcal{F}}_{6,n}(x)\mathbb{H}_{n}(x;q)=0.
\end{equation*}%
Therefore, reagrouping the terms yields the conclusion 
\begin{equation*}
{\mathcal{F}}_{4,n}(x)\Xi _{1,n}(x){\mathscr D}_{q}^{2}\mathbb{H}_{n}(x;q)-{%
\mathcal{F}}_{6,n}(x)\Xi _{1,n}(x){\mathscr D}_{q}\mathbb{H}_{n}(x;q)
\end{equation*}%
\begin{equation*}
+[{\mathcal{E}}_{4,n}(x){\mathcal{F}}_{6,n}(x)-{\mathcal{E}}_{6,n}(x){%
\mathcal{F}}_{4,n}(x)]\mathbb{H}_{n}(x;q).
\end{equation*}
\end{proof}



A similar analysis to that carried out in \cite{hermoso2020second} yields
the following result.



\begin{theorem}[Second order difference equation II]
\label{sordDEqII} Let $\{\mathbb{H}_{n}(x;q)\}_{n\geq 0}$ be the sequence of 
$q$-Hermite I-Sobolev type polynomials defined by \eqref{ASPTSHR}. Then, the
following statement holds, 
\begin{equation*}
\overline{{\mathcal{R}}}_{n}(x){\mathscr D}_{q^{-1}}{\mathscr D}_{q}\mathbb{H%
}_{n}(x;q)+\overline{{\mathcal{S}}}_{n}(x){\mathscr D}_{q^{-1}}\mathbb{H}%
_{n}(x;q)+\overline{{\mathcal{T}}}_{n}(x)\mathbb{H}_{n}(x;q)=0,\quad n\geq 0,
\end{equation*}
where 
\begin{equation*}
\overline{{\mathcal{R}}}_{n}(x)={\mathcal{R}}_{n}(q^{-1 }x),
\end{equation*}
\begin{equation*}
\overline{{\mathcal{S}}}_{n}(x) ={\mathcal{S}}_{n}(q^{-1 }x)+(q^{-1}-1)x{%
\mathcal{T}}_{n}(q^{-1}x),
\end{equation*}
and 
\begin{equation*}
\overline{{\mathcal{T}}}_{n}(x) ={\mathcal{T}}_{n}(q^{-1 }x).
\end{equation*}
\end{theorem}



\begin{proof}
Combining \eqref{DqProp} with \eqref{SDEq1}, and then using \eqref{cadrule}
we get%
\begin{equation*}
q\overline{{\mathcal{R}}}_{n}(x){\mathscr D}_{q}{\mathscr D}_{q^{-1}}\mathbb{%
H}_{n}(xq;q)
\end{equation*}
\begin{equation*}
+{\mathcal{S}}_{n}(x){\mathscr D}_{q^{-1}}\mathbb{H}_{n}(xq;q)+{\mathcal{T}}%
_{n}(x)\mathbb{H}_{n}(x;q)=0,\quad n\geq 0,
\end{equation*}%
Then, using (\ref{DqProOK}), yields%
\begin{equation*}
qq^{-1}\overline{{\mathcal{R}}}_{n}(x){\mathscr D}_{q^{-1}}{\mathscr D}_{q}%
\mathbb{H}_{n}(xq;q)
\end{equation*}
\begin{equation*}
+{\mathcal{S}}_{n}(x){\mathscr D}_{q^{-1}}\mathbb{H}_{n}(xq;q)+{\mathcal{T}}%
_{n}(x)\mathbb{H}_{n}(x;q)=0,\quad n\geq 0,
\end{equation*}%
Thus%
\begin{equation*}
\overline{{\mathcal{R}}}_{n}(x){\mathscr D}_{q^{-1}}{\mathscr D}_{q}\mathbb{H%
}_{n}(xq;q)
\end{equation*}
\begin{equation*}
+{\mathcal{S}}_{n}(x){\mathscr D}_{q^{-1}}\mathbb{H}_{n}(xq;q)+{\mathcal{T}}%
_{n}(x)\mathbb{H}_{n}(x;q)=0,\quad n\geq 0,
\end{equation*}%
Next, replacing $x$ by $q^{-1}x$ we have%
\begin{equation*}
\overline{{\mathcal{R}}}_{n}(q^{-1}x){\mathscr D}_{q^{-1}}{\mathscr D}_{q}%
\mathbb{H}_{n}(x;q)
\end{equation*}
\begin{equation*}
+{\mathcal{S}}_{n}(q^{-1}x){\mathscr D}_{q^{-1}}\mathbb{H}_{n}(x;q)+{%
\mathcal{T}}_{n}(q^{-1}x)\mathbb{H}_{n}(q^{-1}x;q)=0,\quad n\geq 0.
\end{equation*}%
Finally, a rearrangement of the terms involved gives%
\begin{equation*}
\overline{{\mathcal{R}}}_{n}(q^{-1}x){\mathscr D}_{q^{-1}}{\mathscr D}_{q}%
\mathbb{H}_{n}(x;q)
\end{equation*}
\begin{equation*}
+\lbrack {\mathcal{S}}_{n}(q^{-1}x)+(q^{-1}-1)x{\mathcal{T}}_{n}(q^{-1}x)]{%
\mathscr D}_{q^{-1}}\mathbb{H}_{n}(x;q)+{\mathcal{T}}_{n}(q^{-1}x)\mathbb{H}%
_{n}(x;q)=0,\quad n\geq 0.
\end{equation*}
This allows to conclude the result.
\end{proof}



\section{Conclusions and further remarks}



In this paper, we have considered the $q$-Hermite I-Sobolev type polynomials
of higher order, and described important properties related to them such as
several structure relations, their hypergeometric representation derived
from the $q$-Hermite I polynomials, together with a three-term recurrence
relation of their elements and two $q$-difference equation satisfied by
their elements.

In this last section we present some examples and further remarks on these
families of polynomials.

First, let us fix the parameters defining one of the previous families of
orthogonal polynomials with $\alpha=3$, $q=3/5$. We also choose $j=2$ which
entails that the $q$-derivatives of second order evaluated at $\alpha$ are
involved in the definition of the sequence of polynomials.

The first elements in $H_{n}(x;q)$ are given by 
\begin{equation*}
H_0(x;q)=1,\qquad \mathbb{H}_1(x;q)=x,
\end{equation*}
\begin{equation*}
H_2(x;q)= x^2-\frac{2}{5},\qquad H_3(x;q)=x^3-\frac{98}{125}x,
\end{equation*}
\begin{equation*}
H_4(x;q)=x^4-\frac{3332}{3125}x^2+\frac{1764}{15625},
\end{equation*}
\begin{equation*}
H_5(x;q)= x^5-\frac{97988}{78125}x^3+\frac{2541924}{9765625}x,
\end{equation*}

whereas for every $\lambda>0$, the first elements in $\{\mathbb{H}%
_n(x;q)\}_{n\ge0}$ are given by 
\begin{equation*}
\mathbb{H}_0(x;q)=H_0(x;q),\qquad \mathbb{H}_1(x;q)=H_1(x;q),
\end{equation*}
\begin{equation*}
\mathbb{H}_2(x;q)= H_2(x;q),\qquad \mathbb{H}_3(x;q)\approx H_3(x;q)-\frac{%
9.408\lambda(8.707x^2-3.483)}{17.415\lambda+1},
\end{equation*}
\begin{equation*}
\mathbb{H}_4(x;q)\approx H_4(x;q)-\frac{36.679%
\lambda(277.663x^3+8.707x^2-217.687x-3.483)}{5015.349\lambda+1},
\end{equation*}
\begin{equation*}
\mathbb{H}_5(x;q)\approx H_5(x;q)-\frac{123.658%
\lambda(8686.316x^4+277.663x^3-9252.990x^2-217.687x+977.167)}{%
924614.128\lambda+1}.
\end{equation*}

We illustrate this sequence of orthogonal polynomials in the particular case
of $\lambda=3/5$ (see Figure~\ref{fig1}.



\begin{figure}[tbp]
\centering
\includegraphics[width=0.8\textwidth]{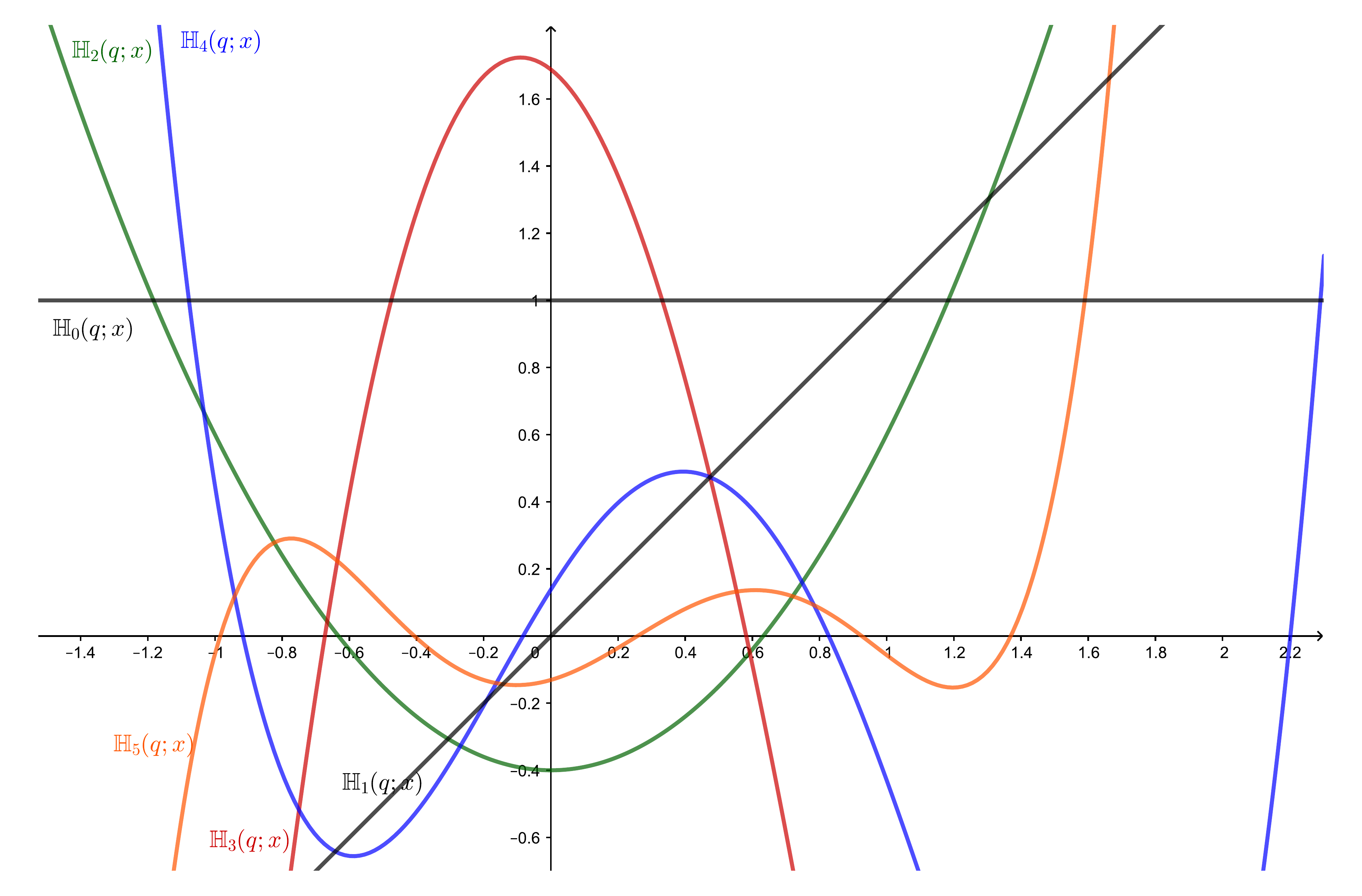} \label{fig1}
\caption{The polynomials $\mathbb{H}_k(x;q)$ for $k=0,\ldots,5$ under the
previous settings}
\end{figure}



Observe that the parity of the elements of the family is not maintained in
the Sobolev settings.

In addition to this, one can check that $\mathbb{H}_k(x;q)$ coincides with $%
H_k(x;q)$ for $k=0,1,2$. This is a general property which is worth remarking
and which comes from the very definition of the polynomials. Indeed, observe
from (\ref{piSob}) that the Sobolev part of the inner product does not
appear when considering functions $f$ or $g$ being polynomials of degree at
most $j-1$ because of the following general property:



\begin{lemma}
\label{e11} Let $p(x)\in\mathbb{R}[x]$ of degree at most $j-1$. Then ${%
\mathscr D}^{j}_{q}p\equiv 0$.
\end{lemma}



\begin{proof}
It is sufficient to prove the result for the monomials $p_k(x)=x^k$ for $%
k=0,1,\ldots,j-1$. We observe that ${\mathscr D}^{j}_{q}1\equiv 0$ and 
\begin{equation*}
{\mathscr D}_{q}x^k=\frac{1-q^{k}}{1-q}x^{k-1},\qquad k\ge 1.
\end{equation*}
This allows to conclude the result.
\end{proof}



Applying Lemma~\ref{e11}, we observe from the definition of the $n$-th
reproducing kernel associated to the $q$-Hermite I polynomials that $%
K_{j-1,q}^{(0,j)}(x,\alpha)\equiv0$. Therefore, in view of (\ref{ConxF1}) we
also obtain that $\mathbb{H}_j(x;q)=H_j(x;q)$.

As a consequence, we have the following result.



\begin{corollary}
$\mathbb{H}_k(x;q)\equiv H_k(x;q)$ for every $0\le k\le j$.
\end{corollary}



Finally, we observe that the $q$-Hermite I polynomials are recovered when $%
\lambda=0$ for the $q$-Hermite I-Sobolev polynomials. Indeed, this can be
checked at many stages described in the work. It is direct to check that for
every $n\in\mathbb{N}$ the polynomial $\mathbb{H}_{n}(x;q)$ tends to $%
H_{n}(x;q)$ for $\lambda\to 0$ in view of (\ref{ConxF1}). Also, the
asymptotic behavior of the elements involved in the connection formulas when 
$\lambda$ approaches to 0 yields the result. It is straightforward to check
the asymptotic behavior of $\mathcal{E}_{k,n}$, $\mathcal{F}_{k,n}$ for $%
k=1,2$, and therefore of $\psi_n(x)$ for $\lambda\to 0$. In addition to
that, this asymptotic behavior can also be observed from the hypergeometric
representation of such polynomials.



\end{document}